\documentclass[preprint,12pt]{elsarticle}

\usepackage{amssymb,amsmath,latexsym,color,amsthm,verbatim,amssymb,amsfonts,amscd, graphicx,multirow, hyperref}
\usepackage[all]{xy}
\usepackage[normalem]{ulem}
\usepackage{enumerate,verbatim}
\usepackage{csquotes}
\usepackage{tikz-cd}

\makeatletter
\def\ps@pprintTitle{%
 \let\@oddhead\@empty
 \let\@evenhead\@empty
 \def\@oddfoot{\centerline{\thepage}}%
 \let\@evenfoot\@oddfoot}
\makeatother

\theoremstyle{plain}
\newcommand\ee{\mathcal{E}}
\newcommand\ff{\mathcal{F}}

\newcommand\ii{\mathcal{I}}
\newcommand\cc{\mathbb{C}}

\newcommand\zz{\mathbb{Z}}
\newcommand\oo{\mathcal{O}}

\newcommand\Hilb{{\rm Hilb}}

\newcommand\pp{{\mathbb{P}}}

\newcommand\po{{\mathbb{P}^1}}

\newcommand\coo{{\mathbb{C}^{1|1}}}
\newcommand\vb{{\,|\,}}
\newtheorem{thm}{Theorem}[section]
\newtheorem{lem}[thm]{Lemma}
\newtheorem{prop}[thm]{Proposition}

\theoremstyle{definition}
\newtheorem{defn}[thm]{Definition}
\theoremstyle{remark}

\newtheorem{ex}{Example}[section]

\journal{Journal of Pure and Applied Algebra}

\begin{document}

\begin{frontmatter}

\title{Families of 0-dimensional subspaces on supercurves of dimension $1\vb 1$}
\author{Mi Young Jang}
\ead{myjang@math.upenn.edu}
\address{Department of Mathematics\\
University of Pennsylvania,\\
 209 South 33rd Street\\
Philadelphia, PA 19104-6395}

\begin{abstract}
In the present paper we prove that the Hilbert scheme of 0-dimensional subspaces on supercurves of dimension $1 \vb 1$ exists and it is smooth. We also show that the Hilbert scheme is not projected in general. 
\end{abstract}


\end{frontmatter}

\tableofcontents

\newpage

\section{Introduction}

Supergeometry is a $\zz_2$-graded generalization of ordinary geometry. The idea of this extension comes from supersymmetry in physics. Instead of sheaves of commutative rings, we consider sheaves of supercommutative rings.
For references, see \cite{berezin,leites,manin}.
Note that superstring perturbation theory can be described as an integral over the muduli space $\mathfrak{M}_g$ of super Riemann surfaces, and if there is a projection map $\mathfrak{M}_g \rightarrow\mathcal{ M}_g$ then we can use the push forward to integrate it. However, $\mathfrak{M}_g$ is non-projected in general.
\begin{prop} (\cite{projected})
The supermoduli space $\mathfrak{M}_g$ is not projected for $g \geq 5$.
\end{prop}
After it has been shown that the supermoduli space is not projected ,
the importance of establishing mathematical foundations about supermanifolds, supermoduli spaces, (analytic) superspaces, etc. has increased.

\bigskip

In this paper, we first show the existence of the (analytic) Hilbert scheme $\Hilb(S)$ of 0-dimensional subspaces on a supercurve $S$ of dimension $1|1$ (see \ref{supercurve} for definition). This Hilbert scheme can be broken up into disjoint union $\Hilb(S)=\bigcup_{(p,q)}\Hilb^{p|q}(S)$ and each $\Hilb^{p|q}(S)$ is a smooth superspace of dimension $p\vb p$. 

\begin{thm}
Let $S$ be a supermanifold of dimension $1\vb 1$. Then the super Hilbert scheme $\Hilb^{p|q}(S)$ is smooth and has dimension $p\vb p$.
\end{thm}

This can be seen as an analogous result to the ordinary case that the Hilbert scheme of $p$ points on a smooth surface is smooth and has dimension $2p$ \cite{fogarty}.

Moreover, perturbative scattering amplitudes of superstring theory is described as integrals over moduli space of super Riemann surfaces with punctures \cite{gidd}.
Note that an element in the super Hilbert scheme $\Hilb^{p\vb q}(S)$ can be viewed as a supercuve $S$ with $(p-q)$ Neveu-Schwarz punctures and $q$  Ramond punctures.

Another interesting related topic that we will not deal with in this paper is Abel's theorem for supercurves. For more details see (\cite{roth}).

\bigskip

The final chapter of this paper is devoted to the (non) splitness of the Hilbert scheme. 
To be specific, the Hilbert scheme $\Hilb^{1|1} \left( S\left( \po, \oo_\po(k)\right) \right)$ is split for any $k$, whereas the Hilbert scheme $\Hilb^{2|1} \left( S\left( \po, \oo_\po(k)\right) \right)$ is not split for all $k \neq 0$. In fact, $\Hilb^{2|1} \left( S\left( \po, \oo_\po(k)\right) \right)$ is not even projected if $k \neq 0$. This also guarantees that any superspace containing $\Hilb^{2|1}  \left( S\left( \po, \oo_\po(k)\right) \right)$ for some nonzero $k$ is not split.

\section{Backgrounds}
\subsection{Supergeometry}

We will review definitions of major terms in this section.

\begin{defn}
A \textit{superspace} is a pair $(S,\oo_S)$ where $S$ is a topological space and $\oo_S=\oo_{S,0}\oplus \oo_{S,1}$ is a sheaf of supercommutative rings which is a locally ringed space. Let $\mathcal{J}$ be the ideal generated by the odd part $\oo_{S,1}$. The bosonic space $S_{b} \subset S$ is defined as the closed subspace $\left(S,\oo_S/\mathcal{J}\right)$.
\end{defn}

From now on, we will only consider the superspaces over $\cc$.

Similar to the ordinary space, We can define locally free sheaves on superspaces. 
The only difference is that they have even and odd ranks. For example, a free sheaf of rank $(p\vb q)$ on a superspace $S$ is $\oo_S^{\,p} \oplus \Pi \oo_S^{\,q}$, where $\Pi \oo_S^{\,q}$ is the parity reversed bundle of $\oo_S^{\,q}$.

A superspace $(S,\oo_S)$ is said to be \textit{split} if there is a locally free sheaf $\ee$ on $S_b$ such that $(S,\oo_S)$ isomorphic to $S(S_{b},\mathcal{E}):=\left(S_{b},\wedge^\bullet \mathcal{E}^\vee \right)$. 
We say a superspace $(S,\oo_S)$ is \textit{locally split} if for any $x\in S$ there is a neighborhood $U$ of $x$ such that $(U,\oo_S|_U)$ is split. Let $m$ be the \textit{dimension} of $S_b$ and let $n$ be the rank of vector bundles that define local split subspaces of $S$. Then the dimension of $(S,\oo_S)$ is defined as $m\vb n$.
We say $S$ is \textit{projected} if it has a projection map from $S$ to its bosonic part $S_b$ so that $\oo_S$ endowed with a $\oo_{S_{b}}$-module structure.

For the rest of this paper, we mainly discuss about analytic superspaces. One basic property of analytic superspace is that, like ordinary analytic spaces, we can take local coordinates. 

\begin{ex} 
An analytic affine superspace 
\[\cc^{m|n}=(\cc^m, \oo_{\cc^{m|n}})=S(\cc^m,\oo_{\cc^m}^{\,n})\]
is one of the simplest examples of a split superspace. Here, $\oo_{\cc^{m}}$ represents the sheaf of analytic functions. The structure sheaf of $\cc^{m|n}$ is given by 
\[\oo_{\cc^{m|n}}=\oo_{\,\cc^m}[ \theta_1,\cdots,\theta_n ]\]
with relations $\theta_i\theta_j = -\theta_j\theta_i$ for all $i$, $j$.

Let $U$ be an open subset of $\cc^m$. For an ideal $I \subset  \oo_{\cc^{m|n}}(U) $, we can define an closed subset $Z(I):=Z \left( I \cap \oo_{\, \cc^m}(U) \right) \subset \cc^m$. The analytic subspace defined by $I$ on U is the superspace $\left( Z(I), \oo_Z:=\oo_U/I \right)$.
\end{ex}

\begin{defn}
An analytic superspace $(S,\oo_S)$ is a superspace which is locally isomorphic to some analytic subspace.
\end{defn}

We say that an analytic superspace $(S,\oo_S)$ is \textsl{smooth} at $x \in S$ if there is an open neighborhood $U$ of $x$ such that  $(U,\oo_S|_U)$ is isomorphic to an open subspace of some analytic affine superspace. An analytic superspace $(S,\oo_S)$ is called smooth if it is smooth at every point in $S$.

A locally split analytic superspace $(S,\oo_S)$ is called a \textit{supermanifold} if $S_b$ is a manifold.
Note that a locally split analytic superspace $(S,\oo_S)$ is smooth if and only if it is a supermanifold. 

\begin{defn}\label{supercurve}
A supercurve is a complex supermanifold of dimension $1\vb n$ for some non-negative integer $n$.
\end{defn}


We will focus on analytic superspaces and will drop \enquote{analytic} for simplicity, and denoting it as superspaces.

Note that the moduli space $\mathfrak{M}_g$ parameterizes super Riemann surfaces of genus $g$, where super Riemann surfaces are supermanifolds of dimension $ 1\vb 1$ with superconformal structure. In general, supercurves need not have superconformal structure. 

\bigskip 

Let $R$ be a supercommutative ring. The Jacobson radical $J(R)$ of $R$ is defined to be the intersection of all maximal ideals of $R$.

\begin{lem} \label{nakayama} \emph{(Nakayama's lemma \cite{lam})}  Let $R$ be a supercommutative ring with the Jacobson radical $J(R) \subset R$. For any finitely generated left $R$-module $M$, $J(R)M=M$ implies $M=0$. 
\end{lem}

\subsection{Super Hilbert Scheme}

In this section, we define super version of Hilbert scheme. 
More careful treatment about the theory of representable functors and super-stacks can be found in \cite{adam, codo2}.

\begin{defn} \par~
\begin{enumerate}[i)]
	\item Let $S$ be a superspace. The \textit{Hilbert functor} ${\mathcal{H}}^{p|q}_{S}$ is the contravariant functor from the category $\mathfrak{S}$ of superspaces to the category of sets defined as follows:

\[{\mathcal{H}}^{p|q}_{S}(B)=
\left\{
\begin{array}{c|c}
  \multirow{4}{*}{ \xymatrix{
\mathcal{Z} \ar@{^{(}->}[r] \ar[d]^\pi
		& S \times B \ar[ld] \\ B} } & \mathcal{Z}  \text{ is a closed subspace of } S \times B \\
    & \text{ such that } \pi_*\oo_\mathcal{Z} \text{ is a locally free }\\
		& \text{ $\oo_B$-module of rank } (p\,|\,q)\text{ and }\quad \\
		&\mathcal{Z} \text{ is finite over } B \qquad \qquad\qquad\
\end{array}
\right\}
\]

 The morphism is defined by the pullback
\[
{\mathcal{H}}^{p|q}_{S}(f)= f^*: {\mathcal{H}}^{p|q}_{S}(B) \rightarrow {\mathcal{H}}^{p|q}_{S}(C)
\]
where $f:C \rightarrow B$ and $B, C \in \mathfrak{S}$.

\item Suppose that the Hilbert functor $\mathcal{H}^{p|q}_{S}$ is representable by the superspace $\Hilb^{p|q}(S)$. We call this the \textit{analytic Hilbert scheme}, abbreviated to the Hilbert scheme.
 
\end{enumerate}

\end{defn}

\begin{ex} The Hilbert functor $\mathcal{H}^{1|1}_{\cc^{1|1}}$ is representable by $\cc^{1|1}$.
\[
\xymatrix{
\mathcal{Z} \ar@{^(->}[r] \ar[d]_{\pi} & \cc^{1|1}_{x \vb \theta }\times \cc^{1|1}_{a \vb \alpha} \ar[dl] \\
\cc^{1|1}_{a\vb \alpha}
}\]
Here, the subscripts define coordinates and $\mathcal{Z}$ is defined by the ideal $(x+a + \alpha \theta)$. This can be checked directly, or as a consequence of the proof of Theorem \ref{hilb}.
\end{ex}

We prove the following theorem in Section $4$.
\begin{thm} \label{hilb} Let $S$ be a supercurve. Then the functor ${\mathcal{H}}^{p|q}_{S}$ is representable by the smooth superspace $\Hilb^{p|q} (S)$ of dimension $p\vb p$.
\end{thm}

Note that the dimension of the Hilbert scheme $\Hilb^{p|q} (S)$ only depends on the even part of the Hilbert polynomial.

\bigskip

\subsection{Obstruction class for splitting}

In this section, we review the definition of an obstruction class which has a critical role in verifying splitness of supermanifolds \cite{dona, projected}.

Consider a supermanifold $S=(M,\oo_S)$. 
Let $\mathcal{J} \subset \oo_S$ be the sheaf of ideals generated by all nilpotents and let $\ee:=(\mathcal{J}/\mathcal{J}^2)^\vee$. 
Let ${\rm Isom}(S,S(M,\ee))$ be a sheaf of local isomorphisms defined by relating an open subset $U \subset M$ to the isomorphisms from $S|_U$ to $S(M,\ee)|_U$. 
Note that ${\rm Isom}(S,S(M,\ee))$ is locally isomorphic to ${\rm Aut} (\wedge^\bullet \ee^\vee) \simeq{\rm Aut} (\wedge^\bullet \ee)$. 
Therefore, for given a supermanifold $S$ which is modeled on $M$ and $\ee$, we get an element $\phi \in H^1 (M, {\rm Aut}(\wedge^\bullet \ee) ) $.
Let $G$ be the set of automorphisms of $\wedge^\bullet \ee$ which act trivially on $M$ and $\ee$.
Since the automorphism induced from $S$ preserves $M$ and $\ee$, and we can say that $\phi \in H^1 (M, G ) $.

Consider the filtration of $S$
\[
M =S^{(0)}\subset S^{(1)}\subset \cdots \subset S^{(n)}=S
\] 
where $S^{(i)}=(M,\oo_S/\mathcal{J}^{i+1})$ and $n=\rm{rank}({\,\ee})$. 

Define $G^{(i)}$ to be the set of automorphisms of $S(M,\ee)$ which are trivial on $S(M,\ee)^{(i-1)}$ for $i=1,2,\cdots,n$.
Observe that $G^{(i)}/G^{(i+1)}$ can be identified with $T_{(-)^i}S \otimes \wedge^{i} \ee^\vee$
where $T_{(-)^i}=T_-$ is an odd tangent space if $i$ is odd and $T_{(-)^i}=T_+$ is an even tangent space if $i$ is even. Moreover, it induces a long exact sequence
\[
\cdots \rightarrow H^1(M,G^{(i+1)}) \rightarrow H^1(M,G^{(i)}) \xrightarrow{\omega} H^1( M,T_{(-)^i}M \otimes \wedge^{i} \ee^\vee) \rightarrow \cdots
\] 

We define $\omega_i$ in a similar manner. Suppose $S^{(i)}$ and $S(M,\ee)^{(i)}$ are isomorphic. Then, since $S^{(i+1)}$ and $S(M,\ee)^{(i+1)}$ are locally isomorphic, the local isomorphism defines the cohomology class $\phi^{(i)} \in H^1(M,G^{(i+1)})$. 	
The i-th obstruction class is defined by 
\[\omega_i:=\omega(\phi^{(i-1)}) \in H^1(M, T_{(-)^i}M \otimes \wedge^{i} \ee^\vee) \]

Note that if $S$ is isomorphic to its split model $S(M,\ee)$, then $\phi^{(i+1)}$ is the image of $\phi^{(i)}$ and thus $\omega_i$ is vanishing, for all $i$. 
In conclusion, a non-vanishing obstruction class $\omega_2$ guarantees the non-splitness of a supermanifold $S$.
In section \ref{nonsplit}, we will use this fact to show the non-splitness of the Hilbert scheme.

We also remark the following lemmas.
\begin{lem}(\cite{manin})
A supermanifold of odd dimension $1$ is always split. 
\end{lem}
\begin{lem}(\cite{projected})
A supermanifold of odd dimension $2$ is split if and only if it is projected.
\end{lem}

\section{Local structure of the Hilbert schemes}

We first show the existence and the smoothness of the Hilbert schemes for special cases in this section and extend it to general cases in Section $4$. 
\smallskip

Let's fix coordinates $(x\vb \theta)$ on $\cc^{1|1}$. Consider a family in $\mathcal{H}^{p|q}_{\coo}(Y)$. 
\[
\xymatrix{
\mathcal{Z} \ar@{^(->}[r] \ar[d]_{\pi} & \cc^{1|1}_{x \vb \theta }\times Y \ar[dl] \\
Y&
}\]

Then by the definition, the pushforward $\pi_*\oo_\mathcal{Z}$ is locally free. In fact, it turns out that $\pi_*\oo_\mathcal{Z}$ is free.

\begin{lem} (\cite{jang}) \label{basis} Let $\mathcal{Y}\subset \cc^{1|1}$ be a closed subspace such that $\dim_{\,\cc}{H^0(\cc^{1|1},\oo_{Y})}$ is $p \vb q$. Then $H^0(\cc^{1|1},\oo_{Y})$ has a basis $1,x,\dots,x^{p-1},\theta, x\theta ,\dots, x^{q-1} \theta$ as a $\cc$-vector space.
\end{lem}

\begin{lem} (\cite{jang}) \label{inverse} Let $X=\left(x_{ij}\right)$ be an $n \times n$ (left) invertible matrix and let $\Gamma=(\gamma_{ij})$ be an $n \times n$ matrix such that $\gamma_{ij}^{\,2}=0$ for each $i$ and $j$, then $X + \Gamma$ is (left) invertible.
\end{lem}

   \begin{prop} \label{gen} Pick $[\mathcal{Z} \xrightarrow{\pi} Y] \in \mathcal{H}^{p|q}_{\cc^{1|1}}(Y)$, then $\pi_*\oo_{\mathcal{Z}}$ is a free $\oo_{{Y}}$-module generated by $1,x,\dots,x^{p-1}, \theta, x\theta ,\dots, x^{q-1}\theta$.
\end{prop}

\begin{proof}

Let $R=\cc[x\vb \theta]$ and let $I \subset R$ be an ideal such that $\dim_\cc R/I = p \vb q$. Lemma \ref{basis} says that $1,x,\cdots,x^{p-1},\theta,x\theta,\cdots,x^{q-1}\theta$ generate $R/I$ as a $\cc$-vector space.

Pick $y \in Y$ and let $\ii$ be the ideal sheaf of $\mathcal{Z}$. Then $I:=(\ii_y+ m_y)/m_y$ can be viewed as an ideal in $R$, where $m_y$ is the maximal ideal of the local ring $\oo_{Y,y}$. Then $\cfrac{(\pi_* \oo_\mathcal{Z})_y}{m_y (\pi_* \oo_\mathcal{Z})_y}$ is isomorphic to $R/I$ and has rank $p\vb q$ as a $\cc$-vector space.
Therefore, $\cfrac{(\pi_* \oo_\mathcal{Z})_y}{m_y (\pi_* \oo_\mathcal{Z})_y}$ is generated by $1,x,\cdots,x^{p-1},\theta,x\theta,\cdots,x^{q-1}\theta$. 

By Lemma \ref{nakayama}, we can find an open neighborhood $U$ of $y$ such that $\pi_* \oo_\mathcal{Z}|_U$ is generated by $1,x,\cdots,x^{p-1},\theta,x\theta,\cdots,x^{q-1}\theta$ as an $\oo_Y|_U$-module. Therefore, $\pi_* \oo_\mathcal{Z}$ is a $\oo_Y$-module with free generators 
\[1,x,\cdots,x^{p-1},\theta,x\theta,\cdots,x^{q-1}\theta.\]

\end{proof}

\subsection{Flattening Stratifications}

Flattening stratifications provide a key step for proving the existence of the ordinary Hilbert scheme(\cite{flat}). 
We demonstrate a super-version of the Fattening stratification that is needed in our situation.

\begin{prop} (Flattening Stratification)\label{flat} Let $X$ and $Y$ be analytic superspaces. Let $\ff$ be a coherent sheaf of modules on $Y \times X$ such that the restriction of the support of $\mathcal{F}$ to each fiber of the projection $Y \times X \rightarrow X$ is zero dimensional. Then for each $(p,q) \in \mathbb{N} \times \mathbb{N}$ we have a locally closed subspace $X_{(p,q)} \subset X$ with the following properties:
\begin{enumerate}[i)]
\item $X= \dot{\cup}_{(p,q)} X_{(p,q)}$
\item $\pi_*\ff|_{X_{(p,q)}}$ is locally free of rank $p \vb q$
\item for any analytic superspace $C$ and a map $f: C \rightarrow X$, the pullback $f^*\ff$ is flat over $C$ if and only if $f$ factors through $C \rightarrow X_{(p,q)} \hookrightarrow X$ for some $(p,q)\in \mathbb{N} \times \mathbb{N}$
\end{enumerate}
\end{prop}

\begin{proof}
Pick $x \in X_b$. Then there are $p,q \in \mathbb{N}$ such that 
\[\dim_{k(x)} \ff_x \times_{\oo_{X,x}} k(x)= p\vb q.\] 
Using Lemma \ref{nakayama}, find a neighborhood $U$ of $x$ such that generators of $\ff_x$ also generate $\ff$ on $U$. Then $\ff|_U$ has $p$ even and $q$ odd generators as an $\oo_{X}|_U$-module and they define the surjection
\[
\oo_U^{\,p} \oplus \Pi \oo_U^{\,q} \xrightarrow{\zeta} \ff|_U \rightarrow 0.
\]

Since $\ff$ is coherent, $\ker{\zeta}$ is also coherent and thus it is finitely generated. 
By shrinking $U$, if necessary, we have an exact sequence
\[
 \oo_U^{\,s}\oplus \Pi \oo_U^{\,t} \xrightarrow{\sigma} \oo_U^{\,p} \oplus \Pi \oo_U^{\,q} \xrightarrow{\zeta} \ff|_U \rightarrow 0 
\]
where the image of $\sigma$ is the kernel of $\zeta$.

Consider a map $f: C \rightarrow X|_U$ and the induced exact sequence
\[
	\oo_C^{\,s} \oplus \Pi \oo_C^{\,t} \xrightarrow{f^*\sigma} \oo_C^{\,p} \oplus \Pi \oo_C^{\,q} \xrightarrow{f^*\zeta} f^* (\ff|_U) \rightarrow 0
\]

Observe that $f^* \left(\ff|_U\right)$ is free of rank $p \vb q$ if and only if $f^*\sigma =0 $.
Let $U_\sigma \subset U$ be the closed subspace of $U$ defined by the ideal $I=\left( \sigma_{ij} \right)_{i,j}$ where $\sigma=(\sigma_{ij}) $ is the matrix representation of $\sigma$. Then we can see that $f^* \ff|_U$ is free of rank $p \vb q$ if and only if $f$ factors through $U_\sigma$.

Therefore, $U_\sigma$ represents the functor $\mathcal{G}_U$ defined by 
\[\mathcal{G}_U(f:C \rightarrow X|_U)=\left\{ f^*\ff \rightarrow C \text{ is flat of rank }(p\vb q) \right\}\] 
We can glue all $U_\sigma$'s with fixed $(p\vb q)$ by the universality of representable functors, and $X_{(p,q)} := \cup_\sigma U_\sigma$ satisfies the required properties.
\end{proof}

A flattening stratification plays a pivotal role in constructing the super Hilbert scheme $\Hilb^{p|q}(\cc^{1|1})$ in
the next section.

\subsection{Defining Equation for the Hilbert Scheme} \label{equ}

Let's fix coordinates 
\[\left( \left( x\vb \theta \right), \left(a_0,\cdots,a_{p-1},b_0, \cdots, b_{q-1} \vb \alpha_0,\cdots,\alpha_{q-1}, \beta_0,\cdots,\beta_{p-1} \right) \right)\] on $\coo \times \cc^{p+q|p+q}$.

Let $\mathcal{Y}$ be the closed subspace of $\coo \times \cc^{p+q|p+q}$ defined by the ideal
	\[{J}=\left( x^p + \sum\limits_{i=0}^{p-1}{a_i x^i} + \sum\limits_{i=0}^{q-1}{\alpha_i x^i} \theta,\ x^q \theta + \sum\limits_{i=0}^{q-1}{b_i x^i}\theta + \sum\limits_{i=0}^{p-1}{\beta_i x^i} \right) \]
and let $\pi: \mathcal{Y} \rightarrow \cc^{p+q\vb p+q}$ be the projection 
\[\xymatrix{
\mathcal{Y}\; \ar@{^(->}[r] \ar[d]_{\pi} & \coo \times \cc^{p+q|p+q} \ar[dl]\\
\cc^{p+q|p+q} &
}\]

Then, according to Proposition \ref{flat}, we can find locally closed subspaces $ \cc^{p+q|p+q} _{(m,n)} \subset  \cc^{p+q|p+q}$ for each $(m,n) \in \mathbb{N}\times \mathbb{N}$ such that $\cc^{p+q|p+q} = \bigcup_{m,n}  \cc^{p+q|p+q}_{(m,n)}$ and each $\cc^{p+q|p+q} _{(m,n)}$ has the universal property.

We follow the proof of Proposition \ref{flat} to show that $\cc^{p+q|p+q}_{(p,q)} \simeq \cc^{p|p}$. 

\begin{thm} $\cc^{p+q|p+q}_{(p,q)}$ is isomorphic to $\cc^{p|p}$. 
\end{thm}

\begin{proof}
Let $\mathcal{Y} \subset \coo \times \cc^{p+q\vb p+q}$ be the closed subspace defined by the ideal 
\[{J}=\left( x^p + \sum\limits_{i=0}^{p-1}{a_i x^i} + \sum\limits_{i=0}^{q-1}{\alpha_i x^i} \theta,\ x^q \theta + \sum\limits_{i=0}^{q-1}{b_i x^i}\theta + \sum\limits_{i=0}^{p-1}{\beta_i x^i} \right) \] 
For simplicity, denote generators of $J$ by 

$f:=x^p + \sum\limits_{i=0}^{p-1}{a_i x^i} + \sum\limits_{i=0}^{q-1}{\alpha_i x^i} \theta \ $ and $ \ g:=x^q \theta + \sum\limits_{i=0}^{q-1}{b_i x^i}\theta + \sum\limits_{i=0}^{p-1}{\beta_i x^i}$.

Apply the long division by $\left( x^q+\sum\limits_{i=0}^{q-1}b_i x^i \right)$ to $f$ and $g$
\[\begin{split}
f=&\left(x^q+\sum_{i=0}^{q-1}b_i x^i\right)\left(x^{p-q}+\sum_{i=0}^{p-q-1}c'_i x^i\right)+\sum_{i=0}^{q-1}d\,'_i x^i +  \sum_{i=0}^{q-1}\gamma_i x^i\theta\\
g=&\left(x^q + \sum_{i=0}^{q-1}b_i x^i\right)\left(\theta + \sum_{i=0}^{p-q-1}\delta_i x^i\right) +\sum_{i=0}^{q-1}\epsilon_i x^i
\end{split}\]
 
Then change coordinate on $\cc^{p+q|p+q}$ to make this form
\[\begin{split}
f=& (x^q+\sum_{i=0}^{q-1}b_i x^i)(x^{p-q}+\sum_{i=0}^{p-q-1}a_i x^i)+\sum_{i=0}^{q-1}c_i x^i +  \sum_{i=0}^{q-1}\beta_i x^i (\theta + \sum_{i=0}^{p-q-1}\alpha_i x^i)\\
g=& (x^q + \sum_{i=0}^{q-1}b_i x^i)(\theta + \sum_{i=0}^{p-q-1}\alpha_i x^i) +\sum_{i=0}^{q-1}\gamma_i x^i
\end{split}\]

Denote $\sum_{i=0}^{p-q-1}a_i x^i, \sum_{i=0}^{q-1}b_i x^i,\cdots$ by $a,b,\cdots$ for simplicity.

According to the proof in Proposition \ref{flat}, there is an open subset $U \subset \cc^{p+q|p+q}$ and an exact sequence
\begin{equation}\label{ses}
\oo_{U}^{\,s} \oplus \Pi \oo_U^{\,t} \xrightarrow{\sigma} 
\oo^{\,p}_U \oplus \Pi \oo^{\,q}_U \xrightarrow{\phi}
\pi_*\oo_{\mathcal{Y}} \big{|}_U\rightarrow 0
\end{equation}
where $\cc^{p+q|p+q}_{(p,q)}$ is defined by the ideal $I:=\left( \sigma_{ij} \right)_{i,j}$ and the  map $\phi$ is defined as 
\[(A_i\vb \mathcal{A}_j)_{i,j} \mapsto \sum_{i=0}^p A_i x^i + \sum_{j=0}^q \mathcal{A}_j x^j \theta. \]

We claim that $I$ is generated by $c_0, \cdots, c_{q-1}$ and $\gamma_0, \cdots, \gamma_{q-1}$.

As a first step, we assert that 
\begin{equation}\label{ker1}
(\sum_{i=0}^{q-1}c_i x^i)\theta + (\sum_{i=0}^{q-1}c_i x^i)(\sum_{i=0}^{p-q-1}\alpha_i x^i) -\sum_{i=0}^{q-1}\gamma_i x^i(x^{p-q}+\sum_{i=0}^{p-q-1}a_i x^i)
\end{equation}
and
\begin{equation}\label{ker2}
(\sum_{i=0}^{q-1}\gamma_i x^i)\theta + (\sum_{i=0}^{q-1}\gamma_i x^i)(\sum_{i=0}^{p-q-1}\alpha_i x^i)
\end{equation}
are contained in the kernel of $\phi$.

Observe that
\begin{equation*}\begin{split}
f&(\theta + \alpha)-g(x^{p-q}+a)=\,c(\theta+\alpha)-\gamma(x^{p-q}+a)\\
&=\,(\sum_{i=0}^{q-1}c_i x^i)\theta + (\sum_{i=0}^{q-1}c_i x^i)(\sum_{i=0}^{p-q-1}\alpha_i x^i) -\sum_{i=0}^{q-1}\gamma_i x^i(x^{p-q}+\sum_{i=0}^{p-q-1}a_i x^i) \\
\end{split}\end{equation*}
and
\begin{equation*}\begin{split}
g&(\theta + \alpha)
=\,\gamma(\theta + \alpha)\\
&=\,(\sum_{i=0}^{q-1}\gamma_i x^i)\theta + (\sum_{i=0}^{q-1}\gamma_i x^i)(\sum_{i=0}^{p-q-1}\alpha_i x^i)\qquad\qquad\qquad\qquad\qquad\quad\quad\ 
\end{split}\end{equation*}

Since $\mathcal{Y}$ is defined by the ideal generated by $f$ and $g$, (\ref{ker1}) and (\ref{ker2}) are in the kernel of $\phi$.
\[((c_0\alpha_0-a_0\gamma_0, \cdots ,\gamma_{q-1},\; \overbrace{0,\cdots,0}^{p-q}\;),(c_0 , \cdots , c_{q-1}))  \in \ker\phi \]
\[((\gamma_0 \alpha_0 , \cdots,\gamma_{q-1}\alpha_{p-q-1},\; \overbrace{0,\cdots , 0}^{q}\;), (\gamma_0 , \cdots , \gamma_{q-1}
))\in \ker\phi \] 
Hence, $\cc^{p+q|p+q}_{(p,q)}\, \cap\ U $ is contained in the closed subspace $ Z\left (\left\{c_i,\gamma_i\right\}_{i=0}^{q-1}\right) \cap U $. By restricting (\ref{ses}) to ${\mathcal{H}} := Z\left (\left\{c_i,\gamma_i\right\}_{i=0}^{q-1}\right) \cap U $ we get 
\[
\oo^{\,s'}_\mathcal{H} \oplus \Pi \oo^{\,t'}_\mathcal{H} \xrightarrow{\sigma_{{\mathcal{H}}}} 
\oo^{\,p}_\mathcal{H} \oplus \Pi \oo^{\,q}_\mathcal{H} \xrightarrow{\phi_\mathcal{H}} 
\pi_* \oo_\mathcal{Y}|_\mathcal{H} \rightarrow 0  
\]

\textit{\underline{Claim}:} $\phi_\mathcal{H}$ is an isomorphism.

Let $\left(A_1,\cdots,A_p \vb \mathcal{A}_1,\cdots,\mathcal{A}_q \right)$ be an element in a section of $\ker{\phi_\mathcal{H}}$.
We can find sections $C_i$'s and $\mathcal{D}_j$'s of $\oo_{\cc^{1|1}\times \cc^{p+q|p+q}}$ such that
\begin{align*}
\sum_{i=0}^{p-1}&A_i x^i +  \sum_{i=0}^{q-1}\mathcal{A}_i x^i\theta \\
&= C f + \mathcal{D} g \\
&= C (x^q+b) (x^{p-q}+a)+C \beta \alpha + C\beta\theta
   +\mathcal{D} \theta(x^q+b) + \mathcal{D} \alpha(x^q+b).
\end{align*}
I.e.,
\begin{equation}\begin{split}\label{ker3}
\sum_{i=0}^{p-1}A_i x^i= C \left( x^q+\sum_{i=0}^{q-1}b_i x^i \right) \left( x^{p-q}+\sum_{i=0}^{p-q-1}a_i x^i \right)
\qquad \qquad\qquad \qquad\\+C\left(\sum_{i=0}^{q-1}\beta_i x^i \right) \left(\sum_{i=0}^{p-q-1}\alpha_i x^i \right)
+\mathcal{D} \left( \sum_{i=0}^{p-q-1}\alpha_i x^i \right) \left( x^q+\sum_{i=0}^{q-1}b_i x^i\right)
\end{split}\end{equation}
and
\begin{equation}\label{ker4}
\sum_{i=0}^{q-1}\mathcal{A}_i x^i=C \left( \sum_{i=0}^{q-1}\beta_i x^i \right)+\mathcal{D} \left( x^q+\sum_{i=0}^{q-1}b_i x^i\right) 
\qquad \qquad\qquad\qquad \qquad
\end{equation}

Comparing the highest degree terms in (\ref{ker3}), we see that $C=0$. Similarly, from (\ref{ker4}) we can check $\mathcal{D}=0$. 
Then $A_i$ and $\mathcal{A}_j$ vanish for all $i$ and $j$, and $\phi_\mathcal{H}$ is an isomorphism. 
Therefore,  $\cc^{p+q|p+q}_{(p,q)}$ is defined by the ideal \[ \left(c_0,\cdots,c_{q-1},\gamma_0,\cdots,\gamma_{q-1} \right) \] 
and thus $\cc^{p+q|p+q}_{(p,q)}$ is isomorphic to $\cc^{p|p}$.

\end{proof}

\begin{thm}\label{local} $\cc^{p|p}$  represents the Hilbert functor $\mathcal{H}^{p|q}_{\cc^{1|1}}$.
\end{thm}
\begin{proof}
Pick a flat family in $\mathcal{H}^{p|q}_{\cc^{1|1}}(X)$.
\[\xymatrix{
\mathcal{X}\ \ar@{^{(}->}[r] \ar[dr]_p & \cc^{1|1} \times X \ar[d] \\
& X
}\]
Then $p_* \oo_\mathcal{X}$ is generated by $1,x,\cdots,x^{p-1}$ and $\theta, x\theta, \cdots, x^{q-1}\theta$, by Proposotion \ref{gen}. Therefore, $\mathcal{X}$ is defined by an ideal \[\left(x^p +\sum_{i=0}^{p-1}c_i x^i + \sum_{i=0}^{q-1} \gamma_i x^i \theta,\, x^q \theta + \sum_{i=0}^{q-1}d_i x^i\theta + \sum_{i=0}^{p-1}\delta_i x^i  \right) \] for some $c_i,d_i, \gamma_i,\delta_i \in H^0\left(X ,\oo_X \right)$ where $c_i,d_i$ are commutative and $\gamma_i,\delta_i$ are anticommutative. Then there is a unique map $\phi : X \rightarrow \cc^{p+q|p+q}$ such that the pull-back of $\pi: \mathcal{X} \rightarrow \cc^{p+q \vb p+q}$ is $p$. Since $p$ is flat, $\phi$ factors through $\cc^{p+q|p+q}_{(p,q)}$. Therefore, the pull-back of $\mathcal{X}$ to $\cc^{p+q|p+q}_{(p,q)}$ is the universal family.
\end{proof}

For the rest of this paper, we fix coordinates 
\begin{equation}\label{coordinate}
(a_0,\cdots,a_{p-q-1},b_0,\cdots,b_{q-1} \vb \alpha_0,\cdots,\alpha_{p-q-1},\beta_0,\cdots,\beta_{q-1})
\end{equation} on the super Hilbert scheme $\Hilb^{p|q}(\cc^{1|1}) \simeq \cc^{p|p}$, so that the ideal of the universal family is generated by 
\[
(x^q+\sum_{i=0}^{q-1}b_i x^i)(x^{p-q}+\sum_{i=0}^{p-q-1}a_i x^i) +  \sum_{i=0}^{q-1}\beta_i x^i (\theta + \sum_{i=0}^{p-q-1}\alpha_i x^i)\] and
\[
(x^q + \sum_{i=0}^{q-1}b_i x^i)(\theta + \sum_{i=0}^{p-q-1}\alpha_i x^i). \]

\section{Families of 0-dimensional subspaces on supercurves}

For ordinary functors, it is well known that a functor is representable if it has an open covering by representable open subfunctors (\cite{stack}).
We can use the same logic to show the representability of the Hilbert functor $\mathcal{H}^{p|q}_S$ for a smooth supercurve $S$.

\bigskip

\begin{proof}\textit{of Theorem \ref{hilb}.}
\smallskip

 Let $U= \dot{\bigcup}_i U_i \subset S$ be a finite disjoint union of open subspaces of $S$ such that $U_i$ is isomorphic to some nonempty open subspace of $\coo$. 
Let $\mathcal{H}^{p|q}_{S,U}$ be the open subfunctor of $\mathcal{H}^{p|q}_S$ defined as
$
\coprod_{ \substack {\sum p_i=p \\ \sum q_i =q} } \prod_i\; \mathcal{H}^{p_i|q_i}_{U_i}
$

Then the Hilbert functor $\mathcal{H}^{p|q}_S$ is the union of open subfunctors $\bigcup_{U} \mathcal{H}^{p|q}_{S,U}$, and each $\mathcal{H}^{p|q}_{S,U}$ is representable by a smooth superspace of dimension $(p|p)$ as an application of Theorem $\ref{local}$.

To be specific, let's consider any family $\mathcal{Z} \subset S \times X$ in $\mathcal{H}^{p|q}_S(X)$. 
For each $x \in X$, we can find a neighborhood $V$ of $x$ such that the support of $\mathcal{Z}|_{\pi^{-1}(V)}$ is contained in $U \times X$ for some $U=\dot{\cup}_i U_i \subset S$. 
Then there is a map from $V$ to $\Hilb^{p|q}(U)$ such that $\mathcal{Z}|_{\pi^{-1}(V)}$ is the pullback of the universal family. 
Let $X=\cup_\alpha V_\alpha$ be an open covering of $X$ constructed as above and let $U_\alpha \subset S $ be the corresponding open subspaces. Then the universality of the Hilbert scheme guarantees that we can glue $\Hilb^{p|q}(U_\alpha)$ for all $V_\alpha$.



 Therefore, the Hilbert functor $\mathcal{H}^{p|q}_S$ is representable by a dimension $(p|p)$ smooth superspace.

\end{proof}

For the ordinary Hilbert scheme of points, the Hilbert scheme $\Hilb^{4}(\cc^{3})$ is not smooth. We can see this by checking the non-smoothness of $\Hilb^{4}(\cc^{3})$ at $I=m^2=(x,y,z)^2$. 
More details about the smoothness or non-smoothness of the Hilbert scheme $\Hilb^{p|q}(\cc^{1|2})$ can be found in my PhD thesis. Actually, it turns out that $\Hilb^{p|q}(\cc^{1|2})$ is not smooth for certain cases.

\section{(Non)splitness of the Hilbert scheme}

In the previous sections, we found local defining equations and gluing maps of the Hilbert scheme $\Hilb^{p|q} C$ where $C$ is a supermanifold of dimension $1 \vb 1$. We use these to construct an obstruction class and show (non) splitness of the Hilbert scheme.

\subsection{A split Hilbert scheme}
\begin{ex}
Consider the line bundle $\oo_\po(k)$ on $\po$. Then the supermanifold $S:=S\left( \po, \oo_\po(k) \right)$ has dimension $1|1$ and the Hilbert scheme $\Hilb^{p|q}(S)$ has dimension $p\vb p$. 

Consider the standard open cover $\pp^1=U_0 \cup U_1$ and assign affine coordinates on each $S|_{U_i}$
\[ S |_{U_0} \simeq \cc^{1|1}_{x,\theta} \, , \qquad S |_{U_1} \simeq \cc^{1|1}_{y,\psi} \]

Observe that $\left( \Hilb^{1|1}(S) \right)_b = \po$ and, due to Theorem \ref{local}, we have $\Hilb^{1|1}(S)|_{U_0}\simeq \cc^{1|1}_{a,\alpha}$ and $\Hilb^{1|1}(S)|_{U_1} \simeq \cc^{1|1}_{b,\beta}$. Then the Hilbert scheme $ \Hilb^{1|1}(S)$ is split since it has odd dimension $1$, but it still has interesting structure.

 From the relations 
\[ x=a+\alpha \theta,\ y=b+\beta \psi,\ y=1/x,\ \psi=\theta/x^k \text{ and } b={1}/{a}\] on the intersection $U_0 \cap U_1$, we can compute the transition map $\beta = -a^{k-2} \alpha$. Therefore, $\Hilb^{1|1}(S) = S(\po,  W)$ where $W=\oo(k -2)=\oo(-2) \otimes \oo(k)$ and $\Hilb^{1|1}(S)$ is split.
\end{ex}

\subsection{The super Hilbert scheme $\Hilb^{2|1}\left( S\left( \po, \oo_\po(k) \right) \right)$}\label{nonsplit}

Consider the supercurve $S:=S\left( \po, \oo_\po(k) \right)$. Note that the bosonic part of $\Hilb^{2|1}\left( S \right)$ is $\po \times \po$. 

Let $\Delta \subset \po \times \po$ be the diagonal. Let $U_{ij}=U_i \times U_j \subset \po_{[z_0;z_1]} \times \po_{[w_0;w_1]}$ be the open subset where $U_i$ is defined by $z_i \neq 0$ and $U_j$ is defined by $w_j\neq 0$.  

Consider the open cover $\po \times \po = \bigcup_{k=1}^4 V_i$, where 
\[ V_1:=U_{00},\quad V_2:= U_{10}-\Delta, \quad V_3:=U_{01}-\Delta\quad \text{and} \quad V_4:=U_{11}.\]

Then we can see that the Hilbert scheme $\Hilb^{2|1}(S)$ can be covered by four open subsets
\[
\Hilb^{2|1}(S)=\bigcup_{k=1}^4 \Hilb^{2|1}(S)\big{|}_{V_k} 
\]

Let $p_{10}$ and $p_{01}$ to be the projections to the reduced parts
\[p_{10}: \Hilb^{1|1}(S |_{U_1}) \times \Hilb^{1|0} (S|_{U_0}) \rightarrow U_1 \times U_0 \subset \po \times \po\] 
\[p_{01}: \Hilb^{1|1}(S |_{U_0}) \times \Hilb^{1|0} (S|_{U_1}) \rightarrow U_0 \times U_1 \subset \po \times \po\] 
Note that, since $\Hilb^{1|1}(S |_{U_i})$ and $\Hilb^{1|0}(S |_{U_i})$ are isomorphic to the affine space $\cc^{1|1}$, we can take the natural projection to the reduced parts. 

Let  $\Delta^* := p^* \Delta$ be the pullback of the diagonal for each $p=p_{10},p_{01}$. 

First, note that we can naturally identify
\begin{align*}
\Hilb^{2|1}(S) \big{|}_{V_2} \simeq \; \Hilb^{1|1}(S|_{U_1}) \times \Hilb^{1|0}(S|_{U_0}) -\Delta^* \\
\Hilb^{2|1}(S) \Big{|}_{V_3} \simeq \; \Hilb^{1|1}(S|_{U_0}) \times \Hilb^{1|0}(S |_{U_1}) -\Delta^*
\end{align*}

Assign coordinates as in (\ref{coordinate})
\begin{align}
S|_{U_0} \simeq \; \cc^{1|1}_{x,\theta}\quad\, &\\
S|_{U_1}  \simeq\; \cc^{1|1}_{y,\psi}\quad\, &\\
\Hilb^{2|1}(S)|_{V_1} \simeq &\; \cc^{2|2}_{a_1,a_2 \vb \alpha_1,\alpha_2}\label{a} \\ 
\Hilb^{2|1}(S) \big{|}_{V_2} \simeq &\; \Hilb^{1|1}(S|_{U_1}) \times \Hilb^{1|0}(S|_{U_0}) -\Delta^* \label{b1}\\
\simeq & \; \cc^{1|1}_{b_1|\beta_1}\times \cc^{1|1}_{b_2|\beta_2} -\widetilde{\Delta}\label{b2}\\
\Hilb^{2|1}(S) \Big{|}_{V_3} \simeq&\; \Hilb^{1|1}(S|_{U_0}) \times \Hilb^{1|0}(S |_{U_1}) -\Delta^*\\
\simeq & \; \cc^{1|1}_{c_1\vb \gamma_1} \times \cc^{1|1}_{c_2\vb \gamma_2} - \widetilde{\Delta} \label{c}\\
\Hilb^{2|1}(S)|_{V_4} \simeq &\; \cc^{2|2}_{d_1,d_2 \vb \delta_1,\delta_2} \label{d}
\end{align}
where $\widetilde{\Delta}$ is defined by  $b_1b_2=1$ in (\ref{b2}) and $c_1c_2=1$ in (\ref{c}).

Observe that the isomorphism (\ref{b2}) and (\ref{c}) are given by
\begin{align*}
\cc^{1|1}_{b_1|\beta_1}\times \cc^{1|1}_{b_2|\beta_2} -\widetilde{\Delta} 
& \rightarrow \Hilb^{1|1}(S|_{U_1}) \times \Hilb^{1|0}(S|_{U_0}) -\Delta^* \\
\left( (b_1|\, \beta_1),(b_2|\,\beta_2) \right) \qquad
 &\mapsto \left< y+b_1+ \beta_1 \psi \right> \times \left< x+b_2, \theta + \beta_2 \right>\\
&\rightarrow \Hilb^{2|1}(S) \Big{|}_{V_2}\\
&\mapsto \left< \left(y+b_1+ \beta_1 \psi \right) \left(x+b_2\right), \left(y+b_1+ \beta_1 \psi \right)( \theta + \beta_2 ) \right>
\end{align*}
and
\begin{align*}
\cc^{1|1}_{c_1\vb \gamma_1} \times \cc^{1|1}_{c_2\vb \gamma_2} - \widetilde{\Delta} 
&\rightarrow \Hilb^{1|1}(S|_{U_0}) \times \Hilb^{1|0}(S |_{U_1}) -\Delta^* \\
\left( (c_1|\, \gamma_1),(c_2|\,\gamma_2) \right) \qquad
 &\mapsto \left< x+c_1+\gamma_1\theta \right> \times \left< y+c_2, \psi + \gamma_2 \right>\\
&\rightarrow \Hilb^{2|1}(S) \Big{|}_{V_3}\\
&\mapsto \left< (x+c_1+\gamma_1\theta )(y+c_2),(x+c_1+\gamma_1\theta )(\psi+\gamma_2) \right>.
\end{align*}

On the intersection $V_1 \cap V_3$, by using the condition $c_2\neq 0$ and identities $y=\frac{1}{x}$ and $\psi =\frac{\theta}{x^k}$, we get\begin{equation}\label{glue}
\begin{split}
&\left< (x+c_1+\gamma_1\theta )(y+c_2),(x+c_1+\gamma_1\theta )(\psi+\gamma_2) \right>\\
& =\left< (x+c_1+\gamma_1 \theta)(x+\frac{1}{c_2}),(x+c_1+\gamma_1\theta)(\theta+\frac{\gamma_2}{(-c_2)^k}) \right> \\
& = \left< \left(x+ c_1 - \frac{\gamma_1\gamma_2}{(-c_2)^{k}}\right) (x+c_2^{-1}) 
+ \gamma_1 (c_2^{-1}-c_1) \left( \theta+\frac{\gamma_2}{(-c_2)^{k}}\right),\right. \\
&\left. \qquad\qquad\qquad\qquad\qquad\qquad \quad
\left(x+ c_1 - \frac{\gamma_1\gamma_2}{(-c_2)^{k}}\right) \left( \theta+\frac{\gamma_2}{(-c_2)^{k}}\right) \right>.
\end{split}\end{equation}
Then the gluing map on $V_1 \cap V_3$ is given by 
\begin{equation}
(a_1,a_2 \vb \alpha_1,\alpha_2) =
 \left(c_1 - \gamma_1 \gamma_2 (-c_2)^{-k},\, \frac{1}{c_2} \, \bigg{|} \, \gamma_1\left(\frac{1}{c_2}-c_1\right),\, \gamma_2(-c_2)^{-k}\right)\label{on13}\end{equation}
Similarily, we can compute the gluing map on $V_1 \cap V_2$
\begin{equation}
(a_1,a_2 \vb \alpha_1,\alpha_2) = 
\left(\frac{1}{b_1} + \beta_1\beta_2(-b_1)^{k-2},\, b_2 \, \Big{|}  - \beta_1(-b_1)^{k-2}(b_2-\frac{1}{b_1}) ,\, \beta_2 \right)
\label{on12}\end{equation}
The gluing maps on $V_2 \cap V_4$ and $V_3 \cap V_4$ can be computed by using symmetry.

Let $W$ be the vector bundle defined by $W^\vee=\mathcal{J}/\mathcal{J}^2$ where $\mathcal{J} \subset \oo_{\Hilb^{2|1}(S)}$ is the ideal sheaf generated by all nilpotents.
To check the (non)splitness of the $\Hilb^{2|1} (S)$, 
it is enough to find the obstruction class  $ \omega_2=w(\varphi^{(1)}) \in {\rm H}^1(\po \times \po,\mathcal{T}_{\po \times \po} \otimes \wedge ^2 W^\vee) $ and check it is vanishing or not. (\cite{projected})

Since $\wedge ^2 W^\vee$ is a line bundle on $\po \times \po$, there are $a$ and $b$ such that 
\[
\wedge ^2 W^\vee \simeq \oo(a,b)
\]
\begin{lem}
$a=k-3$ and $b=-k-1$ \label{lb}
\end{lem}
\begin{proof}

First of all, to compute $a$, restrict $\wedge^2W^\vee $ to $\po \times \{0\}$.
\[
\wedge^2W^\vee \big{|}_{\po \times \{0\}} \simeq \oo_\po (a)
\] 
Then the transition map between $V_1$ and $ V_2$ gives the transition map between $U_0$ and $U_1$, where $U_0$ and $U_1$ are standard open sets on $\po \simeq \po \times \{0\} $. By setting $b_2=0$, the gluing map (\ref{on12}) gives us 
\begin{equation}
\alpha_1\alpha_2 = \beta_1\beta_2(-b_1)^{k-3} \label{k-3}\end{equation} 
Note that the section $\alpha_1 \alpha_2$ generates the line bundle $\wedge^2 W^\vee$ on $V_1$ and $\beta_1 \beta_2$ generates the line bundle $\wedge^2 W^\vee$ on $V_2$.
Therefore, (\ref{k-3}) gives us $a=k-3$.

To compute $b$, restrict the line bundle $\wedge^2W^\vee$ to $\{0\} \times \po$. 
Then by plugging in $b_1=0$ to the transition map on $V_2 \cap V_4$, we get
\begin{align*}
\delta_1 &\mapsto \frac{\beta_1}{b_2}\\
\delta_2 &\mapsto (-b_2)^{-k}\beta_2
\end{align*}
Therefore, $\delta_1 \delta_2 = - \beta_1 \beta_2 (-b_2)^{-k-1}$ and $b=-k-1$. 

\end{proof}

\bigskip

We are now ready to prove the main theorem.
Note that every supermanifold of odd dimension $2$ modeled on $M$ and $W$ is uniquely determined up to isomorphism by a cohomology class $\omega \in H^1(M, \mathcal{T}_{M}\otimes\wedge^2 W^\vee)$. Therefore, the Hilbert scheme $\Hilb^{2|1}(S)$ of odd dimension $2$ is split if and only if the obstruction class $\omega$ vanishes.   

\begin{thm} \label{split} Let $S$ be the Hilbert scheme $S(\po, \oo_\po(k) )$.
The Hilbert scheme $\Hilb^{2|1}(S)$ is not split for all $k\neq 0$ and it is split for $k=0$.
\end{thm}
\begin{proof}

\par~
First, note that $\mathcal{T}_{\po \times \po}\otimes \wedge^2 W^\vee$ is the sheaf of $\wedge^2 W^\vee$-valued even derivations on $\po \times \po$.

\begin{enumerate}[i)]

\item The transition map (\ref{on12}) on $V_{12}:= V_1 \cap V_2$

\begin{align*}
a_1 &\mapsto \frac{1}{b_1} + \beta_1\beta_2(-b_1)^{k-2}\\
a_2 &\mapsto b_2\\
\alpha_1 &\mapsto - \beta_1(-b_1)^{k-2}(b_2-\frac{1}{b_1})\\
\alpha_2 &\mapsto \beta_2
\end{align*}

defines a section $\omega_2^{12} \in \Gamma( V_1 \cap V_2, \mathcal{T}_{\po \times \po} \otimes \wedge ^2 W^\vee)$ as 
\[\omega_2^{12}
\; = \; \beta_1\beta_2(-b_1)^{k-2} \cfrac{\partial}{\partial a_1}
\;= \; - \cfrac{\alpha_1\alpha_2}{a_2-a_1}\ \cfrac{\partial}{\partial a_1}\]
Here the identification $\alpha_1\alpha_2 = - (-b_1)^{k-2}(b_2-\cfrac{1}{b_1}) \beta_1 \beta_2$ is used.

\item On $V_{13}:= V_1 \cap V_3$

The transition map (\ref{on13}) defines $\omega_2^{13} \in \Gamma( V_1 \cap V_3, \mathcal{T}_{\po \times \po} \otimes \wedge^2 W^\vee)$
 \[\omega_2^{13}
	=  -(-c_2)^{-k}\gamma_1\gamma_2 \ \cfrac{\partial}{\partial a_1}
	=-\frac{\alpha_1\alpha_2}{a_2-a_1} \ \cfrac{\partial}{\partial a_1}\]

\item On $V_{23} := V_2 \cap V_3$.

 We have $\omega_2^{23}=0$ because $V_{23} \subset V_{12}\cap V_{13}$.

\item  The transition map on $V_{24}:=V_2 \cap V_4$
gives a section $\omega_2^{24} \in \Gamma( V_2 \cap V_4, \mathcal{T}_{\po \times \po} \otimes \wedge ^2 W^\vee)$ as 
\[\omega_2^{24}
\; = \; \frac{\beta_1\beta_2}{(-b_2)^k}\ \cfrac{\partial}{\partial b_1}
\;= \; \frac{\delta_1\delta_2}{d_2-d_1} \ \cfrac{\partial}{\partial d_1}\]

\end{enumerate}

Then non-vanishing of the obstruction class $\omega_2$ can be proven by showing that there is no element $\sigma =( \sigma_i )_i \in \prod_i \Gamma( V_i,\mathcal{T} \otimes \wedge^2 W^\vee )$ such that the boundary map sends $\left(\sigma_i \right)_i$ to $\left( \omega_2^{ij} \right)_{ij}$.

Suppose that there are $\sigma_i$'s such that $\omega_2^{ij}=\sigma_j-\sigma_i$ on each $V_{ij}$.
More specifically, fix coordinates $\left( [z_0;z_1],[w_0;w_1]\right) \in \po \times \po$ and let 
$f(\frac{z_1}{z_0},\frac{w_1}{w_0}),  \bar{f} (\frac{z_1}{z_0},\frac{w_1}{w_0})$, $g(\frac{z_0}{z_1},\frac{w_1}{w_0}),\ \bar{g}(\frac{z_0}{z_1},\frac{w_1}{w_0})$, 
$h(\frac{z_1}{z_0},\frac{w_0}{w_1}),\ \bar{h}(\frac{z_1}{z_0},\frac{w_0}{w_1})$ and 
$l(\frac{z_0}{z_1},\frac{w_0}{w_1}),\ \bar{l}(\frac{z_0}{z_1},\frac{w_0}{w_1})$ be polynomials such that
\begin{equation}\begin{split}\label{sigma}
\sigma_1=&\ f\left( \frac{z_1}{z_0},\frac{w_1}{w_0} \right)\alpha_1\alpha_2\ \cfrac{\partial}{\partial(\frac{z_1}{z_0})} 
  +\bar{f}\left( \frac{z_1}{z_0},\frac{w_1}{w_0} \right)\alpha_1\alpha_2\ \cfrac{\partial}{\partial(\frac{w_1}{w_0})} \\
\sigma_2=&\ g\left( \frac{z_0}{z_1},\frac{w_1}{w_0} \right) \beta_1 \beta_2\ \cfrac{\partial}{\partial(\frac{z_0}{z_1})} 
	+\bar{g}\left( \frac{z_0}{z_1},\frac{w_1}{w_0}\right) \beta_1 \beta_2\ \cfrac{\partial}{\partial(\frac{w_1}{w_0})}\\
\sigma_3=&\ h\left( \frac{z_1}{z_0},\frac{w_0}{w_1} \right) \gamma_1 \gamma_2\ \cfrac{\partial}{\partial (\frac{z_1}{z_0})}
	+\bar{h}\left( \frac{z_1}{z_0},\frac{w_0}{w_1} \right) \gamma_1 \gamma_2\ \cfrac{\partial}{\partial (\frac{w_0}{w_1})}\\
\sigma_4=&\ l\left(\frac{z_0}{z_1},\frac{w_0}{w_1}\right) \delta_1\delta_2\ \cfrac{\partial}{\partial (\frac{z_0}{z_1})}
	+\bar{l}\left( \frac{z_0}{z_1},\frac{w_0}{w_1} \right) \delta_1\delta_2\ \cfrac{\partial}{\partial (\frac{w_0}{w_1})}
\end{split}	
\end{equation}

Observe that
\begin{equation}\begin{split}\label{om12}
 \omega_2^{12}
=&\ (-\frac{z_0}{z_1})^{k-2}\beta_1\beta_2 \cfrac{\partial}{\partial (\frac{z_1}{z_0})} \\
=&\ \sigma_2 -\sigma_1\\
=&\ -f\left( \frac{z_1}{z_0},\frac{w_1}{w_0} \right)\alpha_1\alpha_2\cfrac{\partial}{\partial(\frac{z_1}{z_0})} 
  -\bar{f}\left( \frac{z_1}{z_0},\frac{w_1}{w_0} \right)\alpha_1\alpha_2\ \cfrac{\partial}{\partial(\frac{w_1}{w_0})}\\& \qquad\qquad\qquad\qquad
	+g\left( \frac{z_0}{z_1},\frac{w_1}{w_0} \right) \beta_1 \beta_2\ \cfrac{\partial}{\partial(\frac{z_0}{z_1})} 
	+\bar{g}\left( \frac{z_0}{z_1},\frac{w_1}{w_0}\right) \beta_1 \beta_2\ \cfrac{\partial}{\partial(\frac{w_1}{w_0})}\\
=&\ f \cdot \left( b_2-\frac{1}{b_1} \right)(-b_1)^{k-2} 
\beta_1\beta_2 \cfrac{\partial}{\partial (\frac{z_1}{z_0})}  
-g \cdot \left( \frac{z_1}{z_0} \right)^2 
\beta_1 \beta_2 \cfrac{\partial}{\partial (\frac{z_1}{z_0})}
+\left( -\bar{f} \alpha_1 \alpha_2+ \bar{g}\beta_1\beta_2 \right)\cfrac{\partial}{\partial(\frac{w_1}{w_0})}
\end{split}\end{equation}
By comparing coefficients, we get 
\begin{equation}
\left( -\frac{z_0}{z_1} \right)^k = -g\left( \frac{z_0}{z_1},\frac{w_1}{w_0} \right)
+ f\left( \frac{z_1}{z_0}, \frac{w_1}{w_0} \right) 
\left( \frac{w_1}{w_0}-\frac{z_1}{z_0} \right) \left( - \frac{z_0}{z_1} \right)^k
\label{20}\end{equation}

Similarly, from $\omega_2^{13}$ we get
\begin{equation}
-\left( -\frac{w_1}{w_0} \right)^k = h \left( \frac{z_1}{z_0}, \frac{w_0}{w_1} \right)
-f\left( \frac{z_1}{z_0}, \frac{w_1}{w_0} \right)\left( \frac{w_1}{w_0} -\frac{z_1}{z_0} \right) \left( -\frac{w_1}{w_0}\right)^k
\label{21}\end{equation}
and $\omega_2^{23}$ gives 
\begin{equation}
h\left( \frac{z_1}{z_0} ,\frac{w_0}{w_1} \right) -g\left( \frac{z_0}{z_1},\frac{w_1}{w_0} \right) \left(- \frac{w_1}{w_0} \right)^k \left( -\frac{z_1}{z_0} \right)^k =0
\label{22}\end{equation}

\begin{enumerate}[ \text{Case} I.]

\item  $k>0$

If $k$ is positive, $g\left( \frac{z_0}{z_1},\frac{w_1}{w_0} \right) \left(- \frac{w_1}{w_0} \right)^k \left( -\frac{z_1}{z_0} \right)^k$ has a term with $w_0$ at the denominator for all nonzero $g$. 
Since $h\left( \frac{z_1}{z_0} ,\frac{w_0}{w_1} \right)$ can not have $w_0$ at the denominator, to make the equality (\ref{22}) true, $g$ and $h$ must vanish. Then the equation $(\ref{21})$ implies
\begin{align*}
\left( -\frac{w_1}{w_0} \right)^k &= \;
f\left( \frac{z_1}{z_0}, \frac{w_1}{w_0} \right)\left( \frac{w_1}{w_0} -\frac{z_1}{z_0} \right) \left( -\frac{w_1}{w_0}\right)^k \\
\Rightarrow \qquad  \qquad 1 &= \;
f\left( \frac{z_1}{z_0}, \frac{w_1}{w_0} \right)\left( \frac{w_1}{w_0} -\frac{z_1}{z_0} \right)
\end{align*}
	which is a contradiction.

\item $k<0$

Observe that $g \left( \frac{z_0}{z_1}, \frac{w_1}{w_0} \right) \cdot \left(-\frac{w_1}{w_0} \right)^k \left( - \frac{z_1}{z_0}\right)^k $ has $z_1$ at the denominator for any nonzero $g \neq 0$.
Hence, the equation (\ref{22}) means that $h=g=0$. Then, as in the case $k>0$, the equation $(\ref{21})$
implies
$1 =
f\left( \frac{z_1}{z_0}, \frac{w_1}{w_0} \right)\left( \frac{w_1}{w_0} -\frac{z_1}{z_0} \right)$ which is a contradiction.

\item $k=0$

If $k=0$, (\ref{20}) and (\ref{22}) implies $h=g=-1$ and $f=0$. Furthurmore, from (\ref{om12}), we get $\bar{f}=0$ and $\bar{g}=0$. Then by using symmetry, we can also see that $\bar{h}=0$ and $\bar{l}=0$. 

In conclusion, we have
\begin{align*}
\sigma_1&=0\\
\sigma_2&=-\beta_1 \beta_2\ \cfrac{\partial}{\partial(\frac{z_0}{z_1})} \\
\sigma_3&=-\gamma_1 \gamma_2\ \cfrac{\partial}{\partial (\frac{z_1}{z_0})}\\
\sigma_4&=0
\end{align*}
and $\omega^{ij}_2= \sigma_j -\sigma_i$ for all $i$ and $j$.
Therefore, the obstruction class is vanishing. Therefore, the Hilbert scheme $\Hilb^{2|1} \left( S(\po, \oo_\po )\right)$ is isomorphic to its split model $\wedge^\bullet W$ where $W=\left(\mathcal{J}/\mathcal{J}^2 \right)^\vee$ and $\mathcal{J}\subset \oo_{\Hilb^{2|1} \left( S(\po, \oo_\po )\right)}$ is the ideal generated by nilpotents.
\end{enumerate}

\end{proof}

\section*{Acknowledgements}
\thispagestyle{empty}

I would like to thank Sheldon Katz for many discussions and suggestions. I would also like to thank Ron Donagi for helpful discussions. The author was partially supported by NSF award DMS-12-01089.



\end{document}